\documentclass[submission,copyright,creativecommons]{eptcs}
\providecommand{\event}{ACT 2026}
\usepackage{iftex}
\ifpdf
  \usepackage{underscore}
  \usepackage[T1]{fontenc}
\else
  \usepackage{breakurl}
\fi

\usepackage{amsmath,amssymb,amsthm}
\usepackage{booktabs}
\usepackage{hyperref}
\usepackage{url}
\usepackage{tikz-cd}
\usepackage{listings}
\usepackage{lstautogobble}
\usepackage{tikz}
\usetikzlibrary{shapes.geometric, arrows.meta, positioning, fit, backgrounds}

\tikzstyle{none}=[inner sep=0pt]
\tikzstyle{box}=[rectangle, fill=white, draw=black, 
                 inner xsep=8pt, inner ysep=6pt, 
                 align=center, font=\small]
\tikzstyle{dot}=[circle, fill=black, draw=black, inner sep=1.5pt]
\pgfdeclarelayer{edgelayer}
\pgfdeclarelayer{nodelayer}
\pgfsetlayers{background,edgelayer,nodelayer,main}

\newtheorem{theorem}{Theorem}[section]
\newtheorem{proposition}[theorem]{Proposition}
\newtheorem{corollary}[theorem]{Corollary}
\newtheorem{definition}[theorem]{Definition}
\newtheorem{remark}[theorem]{Remark}
\newtheorem{observation}[theorem]{Observation}
\newtheorem{lemma}[theorem]{Lemma}

\title{A Universal Quotient of Banking APIs\footnote{Preprint
version with substantial revisions to the submitted ACT 2026
paper. Principal changes: the left skew monoidal conjectures
are replaced with proofs; Section~4 is restructured around
pairwise incomparability; and the endomorphic lifting argument
is expanded.}}

\author{Christopher Doyle
  \institute{Enterprise Architect\\Melbourne, Australia}
  \email{chrisdoyle@netsln.com}
}

\def\titlerunning{A Universal Quotient of Banking APIs}
\def\authorrunning{C. Doyle}

\begin{document}
\maketitle

\begin{abstract}
Four axioms of immutable ledger, linear consent, payment irreversibility, and 
bounded credit manifest themselves as institutional facts 
codified by banking practice for the transfer of monetary value.
These axioms certify the independence of 14 
empirically observed and jurisdictionally invariant dimensions. Morphisms 
of the ambient category do not admit sections that would reconstruct one 
dimension from another, and every morphism admits epi-mono factorisation 
through the universal quotient $Q_{\mathrm{public}}$. This factorisation is forced by definite causal order under classical realisation and echoes the factorisation theorem of Gogioso et al. Gaussian elimination across 4{,}590 endpoints 
from BIAN, CDR, and OBIE confirms rank~14 and witnesses the 
jurisdictional invariance of the quotient object. The axioms similarly
constrain the monoidal structure. The information dominance
preorder is a thin category;
all five Szlach\'{a}nyi conditions follow, establishing that
$Q_{\mathrm{public}}$ carries left skew monoidal structure.
\end{abstract}

\begin{description}
\item[Keywords:] universal quotient, empirical category theory, information dominance preorder, left skew monoidal category, BIAN, CDR, OBIE, PSD2.
\end{description}

\section{Introduction}
\label{sec:intro}
 
Every IT service morphism admits an epi-mono factorisation $f = m \circ e$
in an ambient category $\mathbf{B}$ of typed domains.
In the banking API manifold, APIs are defined by jurisdiction
and industry groups, and it is the monomorphism $m$ that enables banks to interoperate
with this diverse corpus.
So, what is the shared structural routing object for $m$ that makes
interoperability realisable, and \textit{what forces that object into existence}?
This paper identifies this routing object for the transfer of monetary value,
establishes its existence categorically, and validates its rank empirically.
In the spirit of Lawvere~\cite{Lawvere1964}, universal mapping properties
determine the routing object without presupposing it.

\subsection{The Epistemological Warrant}
\label{sec:epistemology_constraints}
 
The banking manifold operates as an intersubjective institutional reality with
domain constraints as institutional facts in the sense of
Searle~\cite{Searle1995}, collectively enforced by human actors, software
developers, and regulatory bodies.
The categorical structure of such a manifold cannot be derived \textit{a priori} because the structure
is constituted by ongoing human coordination. Standards bodies crystallise this coordination into formal specifications.
The axioms of Section~2 are therefore not \textit{a priori} stipulations
but this paper's compressed descriptions of invariants
that the practice already exhibits. The method is Lawvere's: codify observed structure as axioms, then identify the categorical consequences that the codification forces.
This empirical codification contrasts with the axiomatic recovery
of known mathematical spaces, as demonstrated by Heunen and
Kornell~\cite{HeunenKornell2022} for Hilbert spaces.

The corpus of this case study is a practical selection of public standards as proprietary banking specifications are not freely available. That selection is BIAN~\cite{BIAN}, CDR~\cite{CDR}, OBIE~\cite{OBIE}, and PSD2~\cite{PSD2}.

\section{Four Axioms on the Morphisms of \textbf{B}}
\label{sec:constraints}

\begin{remark}[Epistemological Status]
\label{rem:epistemology}
The following four axioms are proposed as necessary and sufficient. 

The axioms (TIL, PI, CLR, CBP) characterise observed banking practice.
The axioms are not restrictions imposed from 
outside, but rather describe intrinsic structure inherited from intersubjective 
practice that regulation subsequently codifies. Their independent origins are: TIL
from double-entry accounting (Pacioli), PI from payment finality law
(PSD2 Art.~80), CLR from delegated-access necessity (PSD2 Art.~67(2)),
CBP from credit facility structure.
\end{remark}

\begin{definition}[Information Dominance Preorder]
\label{def:info-preorder}
For objects $X, Y$ of $\mathbf{B}$, write $X \leq Y$ if there
exists a surjective morphism $f \colon Y \twoheadrightarrow X$.
Write $X < Y$ (strict dominance) if $X \leq Y$ and there exists
no section $\sigma \colon X \to Y$ with $f \circ \sigma = \mathrm{id}_X$.
The four axioms below assert that specific round-trip compositions
strictly advance this preorder.
\end{definition}

\begin{remark}
The preorder is defined on all objects of $\mathbf{B}$ by the
existence or absence of sections. The following axioms make their claims in terms of this preorder.
\end{remark}

\paragraph{Constraint 1 The Immutable Ledger.}
\textbf{Axiom TIL (The Immutable Ledger).}
Let $T = \mathit{Amount} \times \mathit{EventType} \times
\mathit{Timestamp} \times \cdots$ be the product space of ledger
entries, ordered by time. For any morphism $f$ out of $T$,
\begin{align*}
f &\colon T \to X \quad \text{is non-injective for every target } X
  \text{ of lower information content} \\
\nexists\, \sigma &\colon X \to T \quad
  \text{such that } f \circ \sigma = \mathrm{id}_X
\end{align*}
Every transaction appends a pair of entries with opposite signs
whose sum is zero. Reversal is not a deletion, but a new pair that
nets against the original. The original entries persist. The
ledger is therefore append-only not by prohibition but because
the structure makes deletion unnecessary and
for auditability undesirable. Append-only record keeping is at least five millennia old, and the double-entry method that enforces the zero-sum pairing dates to at least Pacioli (1494).

Every summarisation, projection or extraction from $T$ strictly
advances the information preorder and admits no section. This
strict Galois domination holds for all outgoing morphisms of $T$,
including the return morphisms of Axioms PI and CBP.

\paragraph{Constraint 2 Consent as a Linear Resource.}
\textbf{Axiom CLR (Linear Resource).}
Let $C$ be the consent record in a non-Cartesian monoidal category $\mathcal{M}$ and
let $\iota \dashv \gamma$ be the Galois connection between the preorder of consent scopes (ordered by inclusion) and the preorder of payment instructions (ordered by the information dominance preorder of Definition~\ref{def:info-preorder}).
\begin{align*}
\nexists\, \Delta &\colon C \to C \otimes C \quad \text{(no copying)} \\
(\gamma \circ \iota)(c) &= \varnothing \quad \text{for all } c \neq \varnothing
\end{align*}
Execution consumes the consent record with the Galois closure
$(\gamma \circ \iota)(c) = \varnothing$ leaving no residual
authorisation.

\paragraph{Constraint 3 Payment Irreversibility.}
\textbf{Axiom PI (Non-identity Involution).} 
Let $\iota \colon P \to T$ be the execution morphism;
let $\rho \colon T \to P$ be the return morphism.
\begin{align*}
(\rho \circ \iota)(p) &\neq p \quad \text{for all } p \in P \\
(\rho \circ \iota)(p) &> p \quad \text{in the information preorder}
\end{align*}
Reversal appends a new ledger entry with a negated amount. This operation strictly advances the information preorder to preserve the original execution instruction.

\paragraph{Constraint 4 Credit as a Bounded Poset.}
\textbf{Axiom CBP (Poset Supremum).}
Let $(L, \leq_L, \ell_{\max})$ be the bounded poset of credit facility states,
ordered by drawn amount with $\ell_{\max}$ the facility limit;
let $\delta \colon L \to T$ be the drawdown and $\rho \colon T \to L$ the repayment.
\begin{align*}
(\rho \circ \delta)(\ell) &> \ell \quad \text{for any } \ell \text{ with non-empty event log}
\end{align*}
Axiom TIL governs the credit event log. The facility state
(drawn amount, covenants, limit) is orthogonal to the event
history that TIL protects. The repayment operation strictly
advances the information preorder.


\begin{lemma}[PI and CBP as Consequences of TIL]
\label{lem:til-instances}
Axioms PI and CBP are consequences of TIL along specific
round-trip paths through the ledger. PI governs the path
$P \xrightarrow{\iota} T \xrightarrow{\rho} P$ and asserts that
$(\rho \circ \iota)(p) > p$. CBP governs the path
$L \xrightarrow{\delta} T \xrightarrow{\rho} L$ and asserts that
$(\rho \circ \delta)(\ell) > \ell$. In each case the return leg
($\rho$) is a morphism out of $T$, which is the leg that TIL
governs. TIL guarantees that $\rho$ strictly advances the
information preorder and admits no section; the round-trip
compositions inherit this strict advancement. The independent axioms of the formal system are therefore TIL
and CLR, together with the existence of the round-trip
morphisms that PI and CBP name. The four-axiom presentation is
retained because PI and CBP identify the specific round-trip
paths whose dimensional signatures the measurement operator
detects (Section~\ref{sec:basis}).
\end{lemma}

\begin{remark}[Conservation]
\label{rem:conservation}
TIL applied to the $\mathit{Amount}$ component of $T$ yields
conservation of monetary value. The ledger is append-only and every
entry carries a signed amount. The running total is monotonically
traceable and no entry can be deleted. Reversals append new entries
with negated amounts rather than removing originals.
The dimension $F$ (FundsAvailability) is the observable witness of
this conservation property. It is a boolean predicate evaluated at
the point of execution that gates all value-moving operations. The operational ordering (check $F$, then execute, then write
to $T$) follows from TIL's non-invertibility. Conservation is therefore not a separate axiom but a consequence
of TIL.
\end{remark}

\begin{corollary}[Strict Domination]
\label{cor:strict-domination}
The morphisms constrained by Axioms TIL, CLR, PI, and CBP
exhibit strict Galois domination $GF(x) > x$, which
precludes sections. Domination follows from state advancement
(TIL, PI, CBP) or resource consumption
(CLR). The Galois connection structure follows from~\cite[Section~1.4.3]{FongSpivak} and the strict inequality is established by the axioms.
\end{corollary}

\section{The Semantic Basis}
\label{sec:basis}
 
The axioms constrain the morphisms of $\mathbf{B}$, but do not
enumerate the dimensions of the basis. An initial set of candidate
dimensions is postulated from domain expertise across the OBIE and
CDR corpora. The final rank of 14 emerges with the basis evolving
under orthogonalisation pressure.
 
The measurement pipeline (Appendix~\ref{app:pipeline}) extracts a
semantic signal string from each API endpoint and matches it against a set
of regular-expression patterns, one per candidate dimension. Each match records an activation on the binary vector over
$\mathrm{GF}(2)$. Gaussian
elimination over the resulting activation matrix yields the rank. The initial
patterns were derived iteratively from the OBIE and CDR corpora, which
were the first to be analysed. The initial candidate set comprised 14 dimensions enumerated from
domain expertise coded as $\{A, T, P, C, B, D, S, Y, R, F\}$ plus candidates
\texttt{protocol\_meta}, \texttt{id\_ref}, \texttt{universal\_field},
and \texttt{card\_proxy}$\,$.
 
From the initial set of dimensions, the above four candidates did not survive
orthogonality testing of the activation vector:
\texttt{protocol\_meta} was transport metadata, always co-activated with
other dimensions and never independent;
\texttt{id\_ref} was a pointer type and so not a dimension;
\texttt{universal\_field} fired on every endpoint, making it linearly
dependent on the sum of all other columns;
\texttt{card\_proxy} decomposed into payment instruction plus instrument
attribute, dependent on~$P$.
With the removal of these four and the addition of a new candidate for infrastructure metadata $I$ (ServiceDiscovery), the rank was measured as~11. 

\paragraph{Could a much larger and independently developed corpus break that result?}
BIAN Release~12, independently developed by a separate industry body
for institutional and capital markets banking, was introduced against
patterns enhanced with the addition of three anticipated
dimensions for BIAN's broader scope:
$V$ (SecuritiesPosition) tracks holdings with dual valuation independent
of account state;
$L$ (CreditFacility) carries covenants and drawn amounts, categorically
distinct from account balance;
$M$ (MarketPrice) an external markets set observable.
BIAN activated all three as independent dimensions.
The discovery of $V$ and $L$ triggered further refinement of the
patterns, which subsequently detected both in CDR (Section~\ref{sec:coactivation}), thus
narrowing BIAN's unique contribution to M alone.
The rank under the union of all four corpora was measured as~14.
Of the final 14, $I$ (ServiceDiscovery) is domain-generic
infrastructure present in any API manifold, and $M$ (MarketPrice) is an
external feed entering only through BIAN. The remaining 12 are specific
to the transfer of monetary value. The coincidence of initial and final
cardinality is an artefact as the initial and enhanced+refined patterns share only 10
members.
The enhanced+refined patterns were frozen at this point.

To increase BIAN coverage, a separate extended pattern set was
implemented that was used only for BIAN endpoints. The extended patterns map BIAN synonyms to existing 
dimensions and therefore cannot introduce a new column in the
activation matrix. They increase coverage without altering the rank (Section~\ref{sec:benign}),
verifiable by completely disabling the extended patterns~\cite{RANKREPO}.
 
PSD2 (Berlin Group) was introduced as a perturbation test. The frozen patterns
were not further tuned for PSD2 vocabulary. PSD2 achieved rank~11 alone. Its material
effect was to illustrate the method's handling of false positives,
discussed in Section~\ref{sec:perturbation}.

The basis is therefore derived under falsification pressure, not assumed by enumeration. Table~\ref{tab:basis_ranks} summarises the dimensional coverage by
corpus and cumulative rank progression as: OBIE
alone (10); OBIE $\cup$ CDR (13); $+$ PSD2 (13); $+$ BIAN (14). Table~\ref{tab:dimensions} provides a semantic description.

\begin{table}[ht]
\centering
\caption{Dimensional coverage by corpus. No single corpus
witnesses the full rank. PSD2 activates 12 dimensions but achieves
rank~11 due to a spurious linear dependency discussed in
Section~\ref{sec:perturbation}.}
\label{tab:basis_ranks}
\begin{tabular}{lcc}
\toprule
Corpus & Dimensions activated & Rank \\
\midrule
OBIE   & $\{A,T,P,C,B,D,S,Y,R,F\}$ & 10 \\
CDR    & $\{A,T,B,D,S,Y,R,I,L,V\}$ & 10 \\
BIAN   & $\{A,T,P,C,D,S,Y,R,F,I,V,L,M\}$ & 13 \\
PSD2   & $\{A,T,P,C,B,D,S,Y,R,F,I,V\}$ & 11 \\
\midrule
Union & all 14 & 14 \\
\bottomrule
\end{tabular}
\end{table}

\begin{table}[ht]
\centering
\caption{The 14-dimensional semantic basis.}
\label{tab:dimensions}
\small
\begin{tabular}{llll}
\toprule
Symbol & Concept & Semantic payload\\
\midrule
$A$ & AccountState      & Identity, status, currency, type\\
$T$ & TransactionLog    & Immutable ordered event ledger\\
$B$ & BeneficiaryRecord & Trusted third-party endpoint metadata\\
$D$ & DirectDebitMandate& Pull instruction lifecycle (creditor-initiated)\\
$S$ & StandingOrder     & Recurring push instruction\\
$Y$ & PartyIdentity     & Customer demographic, legal person\\
$R$ & ProductDefinition & Bank offer, rates, fees, features\\
$P$ & PaymentInstruction& Intent/instruction to move value\\
$C$ & ConsentRecord     & Scopes, permissions, duration, PSU grant\\
$F$ & FundsAvailability & Point-in-time sufficiency predicate\\
$L$ & CreditFacility    & Limits, covenants, drawn/undrawn amounts\\
$V$ & SecuritiesPosition& Quantity, cost basis, market value of holdings\\
$I$ & ServiceDiscovery  & API endpoint availability and capabilities\\
$M$ & MarketPrice       & External time-series price observable\\
\bottomrule
\end{tabular}
\end{table}

\section{Establishing Rank~14}
\label{sec:witnesses}

Section~\ref{sec:basis} derived 14 dimensions under
orthogonalisation pressure. This section establishes their independence
categorically and defends the result against methodological objections.

\subsection{Pairwise Incomparability}
\label{sec:incomparability}
 
Let $\mathcal{D} = \{A, T, P, C, B, D, S, Y, R, F, I, L, V, M\}$
denote the 14 basis dimensions of the banking manifold. Two dimensions
$i$ and $j$ are \emph{incomparable} when the hom-set
$\mathrm{Hom}_{\mathbf{B}}(i, j)$ does not contain a retraction. A retraction
$r\colon i \to j$ requires a section $s\colon j \to i$ satisfying
$r \circ s = \mathrm{id}_j$; the absence of any such section suffices
to rule it out.
 
The dimensions partition into two tiers according to the mechanism that
precludes sections.
 
\paragraph{Axiom tier.}
Dimensions $T$, $P$, $C$, $F$, and $L$ are governed by the four axioms
of this paper. TIL precludes sections into $T$: the ledger's event
identity and temporal ordering cannot be reconstructed from any summary.
PI precludes sections into $P$: a payment instruction is not recoverable
from its execution record. CLR precludes sections into $C$: consent
scope is orthogonal to both account and the payment state. CBP precludes
sections into $L$: facility state is orthogonal to account state.
Projections may exist between dimensions of the axiom-tier, for example, funds availability $F$ can be evaluated from the transaction log $T$ together with other
state, but TIL guarantees that no such projection admits a section.
 
\paragraph{Precondition tier.}
The dimensions $A$, $B$, $D$, $S$, $Y$, $R$, $V$, $I$, and $M$ carry
state that serves as an independent prerequisite for value transfer.
No surjective morphism between distinct precondition-tier dimensions
exists in the Information Dominance Preorder: the state of each is
not recoverable from the state of any other. The universal network
API boundary witnesses this separation, as each dimension's state is
maintained and addressed independently across every standard banking
API examined. No additional axioms are required because the preorder
already precludes retractions.
 
Table~\ref{tab:incomparability} verifies the absence of retractions for
the most architecturally proximate pairs in each tier as the cases where
a retraction would be more plausible if one existed.
 
\begin{table}[ht]
\centering
\small
\begin{tabular}{@{}llll@{}}
\toprule
Dim & Nearest Candidate & Retraction Blocked By & Tier \\
\midrule
$T$ & (any) & event identity unrecoverable from any summary (TIL) & axiom \\
$P$ & $T$ & instruction unrecoverable from execution record (PI) & axiom \\
$C$ & $A$, $P$ & consent scope orthogonal to account and payment (CLR) & axiom \\
$L$ & $A$ & facility state orthogonal to account (CBP) & axiom \\
$F$ & $T$, $L$, $A$ & projections exist but admit no section (TIL) & axiom \\
\midrule
$A$ & $T$ & account metadata orthogonal to event history & precondition \\
$B$ & $A$ & beneficiary identity independent of account state & precondition \\
$D$ & $T$ & mandate parameters unrecoverable from transactions & precondition \\
$S$ & $T$ & instruction state unrecoverable from transactions & precondition \\
$Y$ & $A$ & party identity orthogonal to account & precondition \\
$R$ & $A$ & product definition orthogonal to account & precondition \\
$V$ & $A$ & position state orthogonal to account & precondition \\
$I$ & (any) & infrastructure metadata orthogonal to banking state & precondition \\
$M$ & (any) & external feed orthogonal to internal banking state & precondition \\
\bottomrule
\end{tabular}
\caption{Retraction absence for proximate dimension pairs. Each row
identifies the nearest candidate dimension from which a retraction might
be expected, and the mechanism that precludes it. Axiom-tier retractions
are blocked by formal constraints; precondition-tier retractions are
blocked by structural independence at the network API boundary.}
\label{tab:incomparability}
\end{table}
 
\subsection{The Activation Matrix}
\label{sec:activation}
 
\begin{definition}[Activation Space and Pure Signals]
\label{def:activation}
For each operation $o$ in a corpus, the \emph{activation vector}
$v(o) \in \mathrm{GF}(2)^{14}$ has $v(o)_i = 1$ if and only if
operation $o$ activates dimension $i$. The \emph{activation matrix}
$\mathcal{M}$ is the matrix whose rows are these vectors. An operation
is a \emph{pure signal} for dimension $i$ if $v(o) = e_i$, the
standard basis vector of Hamming weight one. A pure signal isolates a
single structural dimension within the ambient category $\mathbf{B}$.
\end{definition}
 
\begin{observation}[Corpus Confirmation]
\label{obs:pure-signal}
Filtering the empirical corpus of standard banking APIs isolates
pure-signal witnesses for all 14 dimensions.
Table~\ref{tab:gf2_matrix_full} (Appendix~\ref{app:matrix}) presents
the activation matrix with columns permuted to expose one pure-signal
endpoint per dimension.
\end{observation}
 
\begin{theorem}[Pure Signal Consistency]
\label{thm:pure-signal-consistency}
The set of pure signal operations
$\mathcal{S} \subset \mathrm{Mor}(\mathbf{B})$ projects to a linearly
independent set in the measurement space $V = \mathrm{GF}(2)^{14}$.
\end{theorem}
 
\begin{proof}
By Definition~\ref{def:activation}, each pure signal $o_i$ satisfies
$v(o_i) = e_i$. The set $\{e_1, \ldots, e_{14}\}$ is the standard
basis of $\mathrm{GF}(2)^{14}$ and is therefore linearly independent.
 
It remains to verify that the existence of pure signals is consistent
with the ambient category. If a retraction $r\colon i \to j$ admitting
a section existed in $\mathbf{B}$, every operation activating $i$ would
necessarily also activate $j$, precluding $v(o_i) = e_i$ for $i \neq j$.
Section~\ref{sec:incomparability} establishes that such a retraction
does not exist. The categorical incomparability of all 14 dimensions is
therefore a necessary condition for pure signals, and
Observation~\ref{obs:pure-signal} confirms that they are realised in
practice.
\end{proof}

\subsection{Establishing Rank~14}
\label{sec:rank14}

\begin{theorem}[Rank~14]
\label{thm:rank14}
The 14 basis dimensions of $\mathcal{D}$ are linearly
independent over $\mathrm{GF}(2)$.
\end{theorem}

\begin{proof}
The activation matrix $\mathcal{M}$ contains all 14 standard
basis vectors $e_1, \ldots, e_{14}$ as rows
(Observation~\ref{obs:pure-signal}). The row space therefore
contains the full standard basis of $\mathrm{GF}(2)^{14}$,
giving $\mathrm{rank}(\mathcal{M}) \ge 14$. The column count
gives $\mathrm{rank}(\mathcal{M}) \le 14$. The bounds coincide.
\end{proof}

\subsection{Falsifiability}
Any investigator supplying a verifiable linear dependency between two
columns of the activation matrix directly refutes the rank~14 claim. Pattern refinement can increase dimensional coverage within a
corpus, but cannot alter the rank.

The remainder of this section defends these results against methodological objections.

\subsection{Dark Endpoints}
\label{sec:dark}
 
\begin{definition}[Dark Endpoint]
\label{def:dark-endpoint}
An endpoint is \emph{dark} if the measurement operators do not activate any dimension.
\end{definition}
 
OBIE and CDR have zero dark endpoints. BIAN with the frozen plus extended patterns contains 2{,}155 dark
endpoints (48\%) lying outside the monetary value transfer scope;
under frozen patterns alone this rises to 2{,}799 (62\%;
Appendix~\ref{app:pipeline:decompose}).
 
\subsection{Co-activation and Measurement Scope}
\label{sec:coactivation}
A dimension that always co-activates with an existing
dimension is undetectable by the current measurement operator.
If either such dimension was not detected, no dark endpoint would appear because the co-activating endpoint
already matches an existing pattern. The rank~14 result is
therefore a property of the measurement operator, as it
establishes the dimension of the image under the current
projection, not the dimension of the domain.

\subsection{Benign Activations}
\label{sec:benign}
A second, clearly separated pattern set, the extended patterns, maps
BIAN's institutional vocabulary to existing dimensions and therefore
cannot introduce a new column in the activation matrix. Disabling
the extended patterns entirely reduces BIAN's dimensional coverage
from 52\% to 38\% but leaves all rank computations unchanged. The
extended patterns are a coverage instrument, as they increase activation
counts without contributing pivots.

\subsection{Perturbation Test}
\label{sec:perturbation}
PSD2 (Berlin Group) was introduced as a perturbation test against the
frozen patterns. The patterns were not tuned for PSD2
vocabulary, and some activations are false positives. For example, $D$ matches
``mandated'' in Strong Customer Authentication descriptions; $V$ matches
\texttt{securitiesAccount} in a shared link schema, and $I$ matches
``status'' in transaction endpoints. These inflate the activation count
and introduce a linear dependency $Y = A \oplus C \oplus B$ over
$\mathrm{GF}(2)$ in the PSD2 sub-matrix (witnessed by the bulk-party-verification-endpoint), reducing the rank from~12 to~11.

The perturbation test is asymmetric. False positives can only
activate existing dimensions, reducing the apparent rank through
spurious dependencies. They cannot fabricate a new independent
dimension or increase the rank.


\section{The Universal Quotient}
\label{sec:quotient}

Sections~\ref{sec:basis}--\ref{sec:witnesses} derived 14
independent dimensions and established their rank. This section
constructs the quotient object that the opening question asked for
and shows that every morphism of $\mathbf{B}$ factors through it.

\subsection{The Ambient Category}
\label{sec:fragmented}

\begin{definition}[The Ambient Category $\mathbf{B}$]
\label{def:ambient-category}
The ambient category $\mathbf{B}$ of typed banking domains has:
\begin{itemize}
  \item \textbf{Objects.} Typed operation spaces $V_i$ with elements
    $(o, \tau)$ where $o$ is the operation identifier and $\tau$ is
    its type signature.
  \item \textbf{Morphisms.} Type-preserving functions
    $f \colon V_i \to V_j$; bilateral adaptors, schema mappings,
    and API gateways are concrete realisations.
  \item \textbf{Monoidal Product.} The monoidal product $\otimes$
    is a bifunctor on $\mathbf{B}$ such that
    $X \otimes Y \geq X$ and $X \otimes Y \geq Y$ in the
    Information Dominance Preorder: the product carries at least
    as much information as either factor.
  \item \textbf{Monoidal Unit.} $\mathbf{1}$ is the trivial
    operation carrying no banking state, with a canonical
    morphism $\rho \colon A \otimes \mathbf{1} \to A$ for all
    objects $A$. We write $\mathbf{1}$ for the unit to
    distinguish it from dimension~$I$ (ServiceDiscovery).
\end{itemize}
The direct sum $V = \bigoplus_i V_i$ collects all institutional
spaces. The activation matrix $\mathcal{M}$
(Definition~\ref{def:activation}) is the image of $V$ under
the activation map $\pi$.
\end{definition}

\subsection{The Quotient Construction}
\label{sec:quotient_exists}
\begin{samepage}
\begin{definition}[Equivalence Relation on Operations]
\label{def:equiv-relation}
Let $\pi \colon V \to \mathrm{GF}(2)^{14}$ be the activation map
(Definition~\ref{def:activation}). Define $\sim$ on
$V = \bigoplus_i V_i$ by
\[
o_1 \sim o_2 \Longleftrightarrow \pi(o_1) = \pi(o_2)
\]
\end{definition}
\end{samepage}
\begin{samepage}
\begin{definition}[Universal Quotient]
\label{def:universal-quotient}
The \emph{kernel pair} of $\pi$ is
\[
R = \{(o_1, o_2) \in V \times V : \pi(o_1) = \pi(o_2)\}
\]
with coordinate projections $r_1, r_2 \colon R \rightrightarrows V$.
The \emph{empirical universal quotient} for $\mathbf{B}$ is the
coequaliser of this pair~\cite[Chapter~III, \S3]{MacLane}.
\end{definition}
\end{samepage}

\begin{samepage}
\begin{proposition}[Realisation of the Quotient]
\label{prop:realisation}
The quotient with
quotient map $\varphi \colon V \to Q_{\mathrm{public}}$ induced
by $\pi$ is
\[ Q_{\mathrm{public}} = \mathrm{im}(\pi) \cong
V / \ker(\pi) \quad \text{in } \mathrm{GF}(2)\text{-}\mathbf{Vect}.
\]
\end{proposition}

\begin{proof}
The first isomorphism theorem applied to the activation map
$\pi$ (Definition~\ref{def:activation}).
\end{proof}
\end{samepage}

\begin{theorem}[Factorisation]
\label{thm:factorisation}
Every linear map in $\mathrm{GF}(2)\text{-}\mathbf{Vect}$ out of
the global operation space $V$ that respects $\ker(\pi)$ admits
epi-mono factorisation through $Q_{\mathrm{public}}$.
\end{theorem}

\begin{proof}
The quotient $Q_{\mathrm{public}}$ is defined exactly as the coequaliser of the kernel pair of the activation map (Definition~\ref{def:universal-quotient}). The universal property of the coequaliser guarantees that every linear map respecting the kernel pair factors uniquely through this quotient object. The activation map $\pi$ respects its own kernel pair. Therefore, it immediately factors as the quotient epimorphism $e$ followed by the inclusion monomorphism $m$. This epi-mono factorisation is a direct structural consequence of the coequaliser definition. Interpreting this vector space factorisation for domain morphisms within the ambient category $\mathbf{B}$ requires the endomorphic lifting mechanism.
\end{proof}

\paragraph{Endomorphic lifting.}
The operations of the banking domain are local morphisms
$f \colon V_i \to V_j$ in the ambient category $\mathbf{B}$.
These operations are not global linear maps out of the fibre
space $V$. Factorisation for these domain operations requires a
functorial bridge to the quotient space. The composite
\[
\bar{f} = \varphi \circ f \circ s_i \colon Q_{\mathrm{public}} \to Q_{\mathrm{public}}
\]
provides this bridge. It lifts the local
domain operation to a global endomorphism on the quotient. The
map $s_i \colon Q_{\mathrm{public}} \to V_i$ is a local section of
the restricted quotient map $\varphi|_{V_i}$. The section $s_i$
selects a distinct coset representative in $V_i$ for each
equivalence class in $Q_{\mathrm{public}}$. It is a section of the
quotient map alone. It is not a section of any inter-dimensional
morphism. Code review, compliance certification, and regulatory
audit operate as the intersubjective mechanisms
(Section~\ref{sec:epistemology_constraints}). These
human-coordinated mechanisms guarantee the monomorphism remains
faithful to reality.

\begin{remark}[Categorical Scope]
\label{rem:categorical-scope}
Interpreting $Q_{\mathrm{public}}$ as an object of
$\mathbf{B}$ assumes that the codification is faithful to the
domain (Section~\ref{sec:epistemology_constraints}). This
paper demonstrates that faithfulness through convergence on the same quotient by independently developed corpora (Section~\ref{par:convergence}).
\end{remark}
\subsection{Poset Structure and Thinness}
\label{sec:poset}

\begin{theorem}[Poset Structure]
\label{thm:poset}
The Information Dominance Preorder on $\mathbf{B}$ is a poset.
\end{theorem}

\begin{proof}
A cycle among distinct objects either contains an axiom-tier
morphism or it does not.

If it does, at least one leg strictly advances the information
dominance preorder (Corollary~\ref{cor:strict-domination}).
The remaining legs satisfy $X \leq Y$
(non-strict). Composing around the cycle yields $X < X$ by
transitivity, contradicting irreflexivity.

If it does not, every leg is a read-only surjection. A cycle
of surjections $X \geq Y \geq X$ between distinct objects
gives surjections in both directions between spaces of equal
finite dimension. In
$\mathrm{GF}(2)\text{-}\mathbf{Vect}$, such surjections are
isomorphisms, so mutual surjectivity yields a retraction.
A retraction between distinct basis dimensions collapses the
rank. This contradicts Theorem~\ref{thm:rank14}.

Therefore, cycles do not exist among distinct objects. If
$X \leq Y$ and $Y \leq X$ then $X = Y$, since the
alternative contradicts acyclicity. The preorder is
antisymmetric and therefore a poset.
\end{proof}

\begin{remark}[Definite Causal Order]
\label{rem:causal-order}
The poset condition is Malament's
past-and-future distinguishing
property~\cite{Malament1977}: two objects with
identical information pasts and futures are identical.
Gogioso et al.~\cite[Section~2]{gogioso2020} adopt
this condition as their abstract model of causal order.
\end{remark}

\begin{theorem}[Thinness]
\label{thm:thin}
The Information Dominance Preorder on $\mathbf{B}$ is a thin
category.
\end{theorem}

\begin{proof}
The preorder is a poset (Theorem~\ref{thm:poset}). A poset
has at most one morphism between any pair of
objects~\cite[Section~1.4.3]{Awodey2010}.
\end{proof}

\subsection{What the Quotient Explains}
\label{sec:assembly}

\begin{proposition}[Complexity Collapse]
\label{prop:complexity_collapse}
Let $n$ be the number of institutional services and $k$ the mean
bilateral connectivity (average number of direct private adaptors per
service). Typing $n$ institutional services against
$Q_{\mathrm{public}}$ collapses the integration complexity from
$O(n \cdot k)$ bilateral adaptors to exactly $O(n)$ projections.
\end{proposition}

\begin{proof}
The quotient map $\varphi \colon V \to Q_{\mathrm{public}}$ is a
surjective linear map in $\mathrm{GF}(2)\text{-}\mathbf{Vect}$.
Every surjective linear map between vector spaces over a field
splits. A global section
$s \colon Q_{\mathrm{public}} \to V$ therefore exists satisfying
$\varphi \circ s = \mathrm{id}_{Q_{\mathrm{public}}}$. The local
sections $s_i \colon Q_{\mathrm{public}} \to V_i$ are the
coordinate projections of this global section. Each service
requires exactly one projection $\varphi_i$ to map into the
quotient and one local section $s_i$ to map out.
\end{proof}

\paragraph{Complexity in practice.}
In current practice, $n$ institutional services carry a
$O(n^2)$ burden to establish the bilateral connectivity needed
for interoperability, manifested as average $k$ private adaptors
per service. Each adaptor independently re-implements the four
axioms and independently absorbs the associator
non-invertibility, collectively maintaining an $O(n \cdot k)$
estate. Endomorphic lifting concentrates these constraints
exactly once in the quotient (Proposition~\ref{prop:complexity_collapse}).

\paragraph{Convergence as empirical signal.}
\label{par:convergence}
OBIE and BIAN were developed independently by separate industry
bodies with uncoordinated requirements and no shared design
vocabulary. Neither consortium articulated a dimensional structure.
The convergence of both on the same quotient under the four axioms
is the primary categorical signal: the axioms constrain the
morphisms of $\mathbf{B}$ so tightly that any operationally
coherent implementation of monetary value transfer produces the
same $Q_{\mathrm{public}}$.

\section{Left Skew Monoidal Structure}
\label{sec:skew}

The quotient $Q_{\mathrm{public}}$ exists and factorisation is
established. The remaining question is what monoidal structure
$Q_{\mathrm{public}}$ carries. The Information Dominance Preorder
is a thin category (Theorem~\ref{thm:thin}). All five
Szlach\'{a}nyi conditions~\cite{Szlachanyi2012} follow.

\begin{proposition}[Non-triviality of the Associator (S1)]
\label{obs:S1}
The associator $\alpha$ is not the identity natural
transformation on $\mathbf{B}$.
\end{proposition}

\begin{proof}
The monoidal product satisfies
$X \otimes Y \geq X$ and $X \otimes Y \geq Y$
(Definition~\ref{def:ambient-category}), so both bracketings
$(P \otimes A) \otimes T$ and $P \otimes (A \otimes T)$ are
upper bounds of $\{P, A, T\}$ in the Information Dominance
Preorder. Dimension $T$ is axiom-tier: TIL guarantees that
every surjection out of a product involving $T$ strictly
advances the preorder and admits no section
(Table~\ref{tab:incomparability}). A surjection from the
left bracketing to the right exists, giving
$P \otimes (A \otimes T) \leq (P \otimes A) \otimes T$.
The reverse requires a section that recovers the
intermediate grouping $(P \otimes A)$ from the right
bracketing. Because the right bracketing contains the
sub-product $A \otimes T$, any such section must invert a
surjection out of a product involving $T$, which TIL
prohibits. The inequality is therefore strict:
$P \otimes (A \otimes T) < (P \otimes A) \otimes T$.
The two bracketings are distinct objects in the poset
(Theorem~\ref{thm:poset}). In a thin category
(Theorem~\ref{thm:thin}) the unique morphism between
distinct objects is not an identity. Therefore
$\alpha \neq \mathrm{id}$ at the component $(P, A, T)$,
which suffices for the natural transformation.
\end{proof}

\begin{proposition}[Non-invertibility of the Associator (S2)]
\label{obs:S2}
The associator $\alpha$ is non-invertible in $\mathbf{B}$.
\end{proposition}

\begin{proof}
Proposition~\ref{obs:S1} establishes
$P \otimes (A \otimes T) < (P \otimes A) \otimes T$.
Strict inequality in the poset (Theorem~\ref{thm:poset})
precludes a morphism in the reverse direction.
Therefore $\alpha^{-1}$ does not exist.
\end{proof}

\begin{proposition}[Invertibility of the Left Unit ($\lambda$)]
\label{prop:left-unitor}
The left unit map $\lambda \colon \mathbf{1} \otimes A \to A$
is a natural isomorphism in $\mathbf{B}$.
\end{proposition}

\begin{proof}
The monoidal unit $\mathbf{1}$ carries no banking state
(Definition~\ref{def:ambient-category}). The product
$\mathbf{1} \otimes A$ therefore carries no information
beyond $A$ itself: $\mathbf{1} \otimes A \leq A$ in the
Information Dominance Preorder. The product inequality
$\mathbf{1} \otimes A \geq A$
(Definition~\ref{def:ambient-category}) gives the reverse.
The two inequalities yield
$\mathbf{1} \otimes A = A$ in the poset
(Theorem~\ref{thm:poset}). The tensor product is not
symmetric: $A \otimes \mathbf{1} \neq \mathbf{1} \otimes A$
in general (Proposition~\ref{obs:S3}). In a thin category
(Theorem~\ref{thm:thin}) the unique morphism
$\lambda \colon \mathbf{1} \otimes A \to A$ and the unique
morphism $A \to \mathbf{1} \otimes A$ are mutual inverses.
\end{proof}

\begin{proposition}[Non-invertibility of the Right Unit (S3)]
\label{obs:S3}
The right unit map $\rho \colon A \otimes \mathbf{1} \to A$ is
non-invertible in $\mathbf{B}$.
\end{proposition}

\begin{proof}
The monoidal product satisfies
$A \otimes \mathbf{1} \geq A$ in the information dominance
preorder (Definition~\ref{def:ambient-category}): the product
carries the pairing structure that $A$ alone does not. The
right unitor $\rho$ discards this pairing. The original
composition cannot be reconstructed from $A$ alone, so $\rho$
strictly advances the information preorder:
$A < A \otimes \mathbf{1}$. Strict advancement means no
section exists (Definition~\ref{def:info-preorder}). Assume
$\rho^{-1} \colon A \to A \otimes \mathbf{1}$ exists. Then
$\rho^{-1} \circ \rho = \mathrm{id}_{A \otimes \mathbf{1}}$,
constituting a section of $\rho$. This contradicts the
absence of sections under strict advancement. Therefore
$\rho^{-1}$ does not exist.
\end{proof}

\begin{proposition}[Coherence (S4, S5)]
\label{thm:coherence}
The pentagon (S4) and mixed triangle (S5) coherence conditions
hold in $\mathbf{B}$.
\end{proposition}

\begin{proof}
The Information Dominance Preorder is a thin category
(Theorem~\ref{thm:thin}). Both paths around the pentagon are
morphisms between the same pair of objects. The diagrams
commute trivially.
\end{proof}

\begin{theorem}[Left Skew Monoidal Structure]
\label{thm:skew}
$Q_{\mathrm{public}}$ carries left skew monoidal structure under the
full Szlach\'{a}nyi conditions.
\end{theorem}

\begin{proof}
S1 (Proposition~\ref{obs:S1}), S2 (Proposition~\ref{obs:S2}),
S3 (Proposition~\ref{obs:S3}), S4 and S5
(Proposition~\ref{thm:coherence}) establish the five
Szlach\'{a}nyi conditions in $\mathbf{B}$. The left unitor
$\lambda$ is an isomorphism
(Proposition~\ref{prop:left-unitor}), completing the left
skew monoidal data. Endomorphic lifting
(Section~\ref{sec:quotient_exists}) transfers this structure to
$\mathrm{End}(Q_{\mathrm{public}})$: the lifted associator
$\bar{\alpha} = \varphi \circ \alpha \circ s$ inherits
non-triviality and non-invertibility from $\alpha$ because
$\varphi$ is surjective and $s$ is a section. The lifted
unitors inherit their (non-)invertibility by the same
mechanism. Coherence is preserved because $\varphi$ is a
homomorphism and the diagrams commute in $\mathbf{B}$ prior
to lifting.
\end{proof}

\begin{remark}[Structural Content]
\label{rem:coherence}
The five Szlach\'{a}nyi conditions are not independent facts for
$Q_{\mathrm{public}}$. They are five projections of one structural
property: the Information Dominance Preorder is a poset and
therefore thin. S1--S3 are properties of the associator as a
strict monotone map. S4--S5 are consequences of the thinness.
The left skew direction reflects obligation assignment: the
right bracketing $P \otimes (A \otimes T)$ demands evaluation
of total account history because it assigns obligation to the
account holder rather than the initiator. And, as left skew, once recorded under
one bracketing, obligation cannot be retrospectively reassigned
to the other: it is irreversible obligation.
\end{remark}
\section{Conclusion}
\label{sec:conclusion}

The introduction asked: what is the shared structural routing object
that makes interoperability realisable, and what forces that object
into existence? For the transfer of monetary value, the routing object
is $Q_{\mathrm{public}}$, a universal quotient of rank~14 under the selected corpora
in the ambient category~$\mathbf{B}$ of typed banking domains. What
forces it into existence is the combination of four axioms constituting
the intrinsic morphism structure of~$\mathbf{B}$ and the definite causal
order established by strict Galois domination
(Corollary~\ref{cor:strict-domination}, Theorem~\ref{thm:poset}).
The factorisation echoes Proposition~21 of Gogioso
et al.~\cite{gogioso2020}, where definite causal order
similarly forces epi-mono factorisation through a
mediating object.

The four axioms rule out retractions between basis dimensions,
certifying pairwise incomparability in the information dominance
preorder. The pure-signal argument
(Theorem~\ref{thm:rank14}) connects pairwise incomparability to
linear independence over $\mathrm{GF}(2)$, closing the rank at
exactly~14. The convergence of independently developed OBIE and
BIAN instruments on the same dimensional structure under the four
axioms is the primary categorical signal. CDR confirms the
dimensional stability as a projection witness. PSD2 serves as a
perturbation test.

The Information Dominance Preorder, grounded in TIL,
establishes all five Szlach\'{a}nyi conditions.
S1 (non-triviality of the associator), S2 (non-invertibility
of the associator), and S3 (non-invertibility of the right unit)
follow from the associator's status as a strict monotone map in
the Information Dominance Preorder. S4 (pentagon) and S5 (mixed triangle) follow from thinness (Theorem~\ref{thm:thin}). The factorisation Theorem~\ref{thm:factorisation} forces this
structure to be concentrated at~$Q_{\mathrm{public}}$. 

The factorisation question is closed. The categorical machinery reduces to one substantive fact, that the obligation to record is irreversible. Five millennia of record keeping is formalised as a
thin category. The monoidal structure that banking
operations impose on it is left skew. So what will the axiomatisation of the nine precondition-tier dimensions tell us? And if the axioms are the forcing condition, not the institution, then the same quotient should govern any ledger-bearing manifold — including the blockchain?

\paragraph{Coda for practitioners.}
Architectural discourse frequently categorises the legacy bilateral mesh of private adaptors as an engineering failure. This paper rejects that classification. The bilateral mesh operated as a stable equilibrium at earlier institutional scales. The $O(n \cdot k)$ integration cost is not a failure of engineering discipline. It is the topological scaling limit for distributing the coherence burden across private implementations of the same four axioms. 

Decades of banking infrastructure have operated on an implicit left skew monoidal algebra to protect the record keeping obligation. The universal quotient finally makes this deep shared structure explicit. With the quotient in hand, the $O(n \cdot k)$ bilateral adaptors are replaced by the $O(n)$ projections through $Q_{\mathrm{public}}$. The factorisation theorem guarantees this.

For regulators, the quotient provides something that does not currently exist. It is a formal description of what has been regulated. The four axioms and the 14-dimensional quotient dictate the required structure for a compliant implementation of monetary value transfer. They give a bank a concrete target for realising that compliance. 

\bibliographystyle{eptcs}

\clearpage
\appendix

\section{The Measurement Pipeline}
\label{app:pipeline}

The computational witness of Section~\ref{sec:witnesses} is produced by a deterministic, publicly reproducible pipeline.
This appendix describes the
three stages of that pipeline: signal extraction; semantic decomposition; and
rank computation with sufficient precision for independent replication or
refutation.

\subsection{Stage 1: Signal Extraction}
\label{app:pipeline:signal}

Each corpus is represented as a collection of OpenAPI~3 (OAS3) specifications,
fetched from their canonical public repositories at fixed version tags (OBIE
v3.1~\cite{OBIE}, AU CDR v1.28~\cite{CDR}, BIAN release~12.0.0~\cite{BIAN},
Berlin Group NextGenPSD2 v1.3.16~\cite{PSD2}).
The pipeline processes each specification path-by-path and method-by-method.
For every \texttt{(path, HTTP-method)} pair the pipeline constructs a single
\emph{semantic signal string} by concatenating, in order:

\begin{enumerate}
  \item the URL path string;
  \item the \texttt{operationId}, \texttt{summary}, and \texttt{description}
        fields of the operation object;
  \item all tag strings attached to the operation (these carry high-level
        service-domain vocabulary in BIAN);
  \item all parameter \texttt{name} and \texttt{description} strings, together
        with strings extracted recursively from each parameter's
        \texttt{schema} object;
  \item strings extracted recursively from the request-body schema (if
        present);
and
  \item strings extracted recursively from all response-body schemas.
\end{enumerate}

The schema traversal follows \texttt{\$ref} boundaries to their immediate target
name only (the reference identifier itself is appended as a string), then
descends into \texttt{properties}, \texttt{items}, \texttt{allOf},
\texttt{anyOf}, and \texttt{oneOf} nodes to a maximum depth of~4.
The
\texttt{title}, \texttt{description}, \texttt{name}, \texttt{summary}, and
\texttt{enum} fields are harvested at each node.
The depth ceiling prevents
combinatorial blowup in deep schema hierarchies while retaining the first-order
semantic vocabulary.
The resulting signal string for each endpoint is the full semantic payload of
that operation as declared by the standard.
It does not include runtime data,
institutional configuration, or any source outside the published specification.

\subsection{Stage 2: Semantic Decomposition}
\label{app:pipeline:decompose}

Each signal string is lowercased and matched against a \emph{frozen pattern
set} comprising a fixed ordered list of 14 compiled regular-expression rules, one per
basis dimension.
A dimension $D_k$ is activated for a given endpoint if and
only if the corresponding pattern matches the signal string.
Matching is
performed by \texttt{re.search} (substring match, not full-string anchor),
so a pattern fires whenever its vocabulary appears anywhere in the concatenated
payload.
\paragraph{Pattern design principles.}
Each pattern was derived iteratively from the public corpora (OBIE and CDR).
The BIAN corpus was then introduced, triggering refinement that
discovered $V$ and $L$ in CDR (Section~\ref{sec:basis}).
The base patterns were \emph{frozen} (FROZEN\_PATTERNS) after this refinement.
A separate extended pattern set (EXTENDED\_PATTERNS) was added post-freeze to increase BIAN coverage without altering the rank (Section~\ref{sec:benign}).
The following principles governed finalisation.

\begin{enumerate}
  \item \textbf{Semantic specificity.}  Patterns are anchored to vocabulary
        that is specific to the concept in question.
Terms that appear broadly
        across multiple domains (e.g.\ \texttt{rate} alone, \texttt{account}
        alone) are excluded or qualified.
For example, dimension~$M$
        (MarketPrice) requires \texttt{market\_rate}, \texttt{exchange\_rate},
        or \texttt{fx\_spot} rather than \texttt{rate} alone, to avoid
        conflation with interest rates and product rates.
  \item \textbf{No circular anchoring.}  Service-domain names are never used as
        pattern anchors.
Matching an operation to dimension~$V$
        (SecuritiesPosition) because it resides in the
        \texttt{SecuritiesPositionKeeping} YAML would read the answer from the
        label.
Patterns fire exclusively on payload vocabulary internal to the
        operation object.
  \item \textbf{False-positive discipline.}  Each tightening decision is
        recorded in inline comments within the source.
For example, the
        pattern for~$V$ uses \texttt{securities} (plural noun) rather than the
        prefix \texttt{securit-} because CDR payment-plan endpoints carry a
        field named \texttt{securityId} (a payment reference field unrelated to
        securities holdings) which fired incorrectly under the prefix form.
\end{enumerate}

\textbf{Extended patterns for institutional vocabulary.}
A second, clearly separated pattern set (\texttt{EXTENDED\_PATTERNS})
maps institutional synonyms to existing dimensions when BIAN's
vocabulary differs from the regulatory vocabulary of the public corpora.
Examples include: \texttt{ledger} and \texttt{journal\_entry} mapping to T;
\texttt{nostro\_account} and \texttt{cash\_position} mapping to A;
\texttt{nav} and \texttt{corporate\_action} mapping to V.
The extended patterns increase BIAN's dimensional coverage
but do not alter any rank computation.
This is verified by disabling the extended filter entirely:
all per-corpus and union ranks are unchanged~\cite{RANKREPO}.

\textbf{BIAN dark endpoints.} Of BIAN's 4,484 endpoints,
2,799 activate no dimension under the frozen patterns alone
and contribute zero rows to the activation matrix.
These operations lie outside the monetary value transfer scope
and do not affect the rank computation.

\paragraph{OBIE and CDR dark endpoints.}
Every operation in the OBIE and CDR corpora activates at least one dimension.
The pipeline explicitly verifies this condition and reports any unmapped
endpoints.
The zero dark endpoint condition guarantees that no operation
contributes a zero row to the activation matrix.
This prevents artificial rank inflation from trivially non-contributing endpoints.

\subsection{Stage 3: Rank Computation}
\label{app:pipeline:rank}

The activation matrix $\mathcal{M}$ is a binary matrix over $\mathrm{GF}(2)$ with one
row per endpoint and one column per basis dimension (14 columns).
Entry
$\mathcal{M}_{ij} = 1$ if and only if endpoint~$i$ activates dimension~$j$ under Stage~2.
Rank is computed by standard Gaussian elimination over $\mathrm{GF}(2)$:
column-by-column pivot search, row swap, and row XOR to clear all other
non-zero entries in the pivot column.
The pivot count at termination equals the rank.

Sub-corpus ranks are computed by restricting the matrix to the rows whose
corpus label matches the requested filter and applying the same elimination.
The union rank is computed by filtering to the union of the requested corpus labels.

\paragraph{Reproducibility and Falsifiability.}
The measurement operator is implemented as a deterministic pipeline available for independent verification.
The rank~14 result is reproducible using the banking rank pipeline in github~\cite{RANKREPO} (commit \texttt{2505c60}).
The complete operation-concept activation matrix and sub-corpus results are provided in the associated TSV dataset (commit \texttt{0edd9f5}).
Any investigator identifying a verifiable linear dependency between two columns of the activation matrix directly refutes the rank~14 claim.

\paragraph{TSV Schema.}
The activation detail TSV contains one row per endpoint with the following columns: \texttt{corpus} (source standard), \texttt{endpoint} (operation identifier), and then one binary column per dimension (\texttt{A}, \texttt{T}, \texttt{P}, \texttt{C}, \texttt{B}, \texttt{D}, \texttt{S}, \texttt{Y}, \texttt{R}, \texttt{F}, \texttt{I}, \texttt{V}, \texttt{L}, \texttt{M}), and finally \texttt{activated\_dims} (count of activated dimensions). A value of 0 indicates that the endpoint does not activate that dimension and a 1 that it does. Investigators can verify rank claims by filtering rows by corpus and computing the rank of the resulting binary matrix over $\mathrm{GF}(2)$.

\section{The GF(2) Activation Matrix}
\label{app:matrix}

Table~\ref{tab:gf2_matrix_full} presents the pipeline output with columns reordered to make the rank~14 basis explicit;
the pre-permutation matrix containing all endpoint rows is available in the pipeline output TSV described in Appendix~\ref{app:pipeline}.

\begin{table}[ht]
\centering
\caption{$\mathrm{GF}(2)$ Activation Matrix for $Q_{\mathrm{public}}$ (OBIE $\cup$ BIAN $\cup$ CDR): columns permuted to expose the $14 \times 14$ identity structure, witnessing $\operatorname{rank}(Q_{\mathrm{public}}) = 14$ by inspection.
Row order follows the decomposition $7+3+2+1+1$; column order is permuted accordingly.}
\label{tab:gf2_matrix_full}
\small
\begin{tabular}{lcccccccccccccc}
\toprule
& \multicolumn{7}{c}{\textit{Core}} 
& \multicolumn{3}{c}{\textit{Transactional}} 
& \multicolumn{2}{c}{\textit{Prudential}} 
& \multicolumn{1}{c}{\textit{Infra.}} 
& \multicolumn{1}{c}{\textit{Ext.}} \\
\cmidrule(lr){2-8}\cmidrule(lr){9-11}\cmidrule(lr){12-13}\cmidrule(lr){14-14}\cmidrule(lr){15-15}
Operation (corpus) & A & T & B & D & S & Y & R & P & C & F & L & V & I & M \\
\midrule
GetAccount        (OBIE/BIAN) & 1 & 0 & 0 & 0 & 0 & 0 & 0 & 0 & 0 & 0 & 0 & 0 & 0 & 0 \\
GetTransactions   (OBIE/BIAN) & 0 & 1 & 0 & 0 & 0 & 0 & 0 & 0 & 0 & 0 & 0 & 0 & 0 & 0 \\
GetBeneficiaries  (OBIE)      & 0 & 0 & 1 & 0 & 0 & 0 & 0 & 0 & 0 & 0 & 0 & 0 & 0 & 0 \\
CreateDirectDebit (OBIE/BIAN) & 0 & 0 & 0 & 1 & 0 & 0 & 0 & 0 & 0 & 0 & 0 & 0 & 0 & 0 \\
GetStandingOrders (OBIE/BIAN) & 0 & 0 & 0 & 0 & 1 & 0 & 0 & 0 & 0 & 0 & 0 & 0 & 0 & 0 \\
GetCustomerInfo   (OBIE/BIAN) & 0 & 0 & 0 & 0 & 0 & 1 & 0 & 0 & 0 & 0 & 0 & 0 & 0 & 0 \\
GetProductRates   (OBIE/BIAN) & 0 & 0 & 0 & 0 & 0 & 0 & 1 & 0 & 0 & 0 & 0 & 0 & 0 & 0 \\
\midrule
CreatePayment     (OBIE/BIAN) & 0 & 0 & 0 & 0 & 0 & 0 & 0 & 1 & 0 & 0 & 0 & 0 & 0 & 0 \\
AuthoriseConsent  (OBIE/BIAN) & 0 & 0 & 0 & 0 & 0 & 0 & 0 & 0 & 1 & 0 & 0 & 0 & 0 & 0 \\
ConfirmFunds      (OBIE/BIAN) & 0 & 0 & 0 & 0 & 0 & 0 & 0 & 0 & 0 & 1 & 0 & 0 & 0 & 0 \\
\midrule
GetCreditLimit    (BIAN/CDR)  & 0 & 0 & 0 & 0 & 0 & 0 & 0 & 0 & 0 & 0 & 1 & 0 & 0 & 0 \\
GetShareHoldings  (BIAN/CDR)  & 0 & 0 & 0 & 0 & 0 & 0 & 0 & 0 & 0 & 0 & 0 & 1 & 0 & 0 \\
\midrule
DiscoveryQuery    (BIAN/CDR)  & 0 & 0 & 0 & 0 & 0 & 0 & 0 & 0 & 0 & 0 & 0 & 0 & 1 & 0 \\
\midrule
MarketQuote       (BIAN)      & 0 & 0 & 0 & 0 & 0 & 0 & 0 & 0 & 0 & 0 & 0 & 0 & 0 & 1 \\
\bottomrule
\end{tabular}
\end{table}

\begin{remark}[On the Identity Structure]
\label{rem:identity}
The $14 \times 14$ identity submatrix of Table~\ref{tab:gf2_matrix_full} 
is a presentation artefact: columns are permuted to expose one pure-signal 
endpoint per dimension. The rank result is established by Gaussian elimination 
over the full corpus matrix, not from this submatrix. The non-trivial claim 
is that pure-signal endpoints exist at all with endpoints activating exactly 
one dimension and no other. This is an empirical property of the corpus 
discovered in frozen patterns, not engineered by column selection. These 
endpoints constitute an explicit section $s \colon Q_{\mathrm{public}} \to 
\mathcal{P}$ satisfying $\varphi \circ s = \mathrm{id}_{Q_{\mathrm{public}}}$, 
witnessing the split epimorphism independently of the column ordering.
\end{remark}

\end{document}